\documentclass[12pt]{amsart}
\usepackage{amssymb,amsfonts,latexsym,amscd}

\usepackage[all]{xy}
\usepackage{verbatim}

\newtheorem{theorem}{Theorem}[section]
\newtheorem{lemma}[theorem]{Lemma}
\newtheorem{proposition}[theorem]{Proposition}
\newtheorem{corollary}[theorem]{Corollary}
\theoremstyle{definition}
\newtheorem{definition}[theorem]{Definition}

\newtheorem{example}[theorem]{Example}
\newtheorem{question}[theorem]{Question}

\newtheorem{remark}[theorem]{Remark}

\newcommand{\Tr}{\text{Tr}}

\newcommand{\Rep}{\text{Rep}}
\newcommand{\Vect}{\text{Vec}}

\newcommand{\Qbs}{\overline{\Bbb Q}^\times}
\newcommand{\Qb}{{\overline{\Bbb Q}}}
\newcommand{\Zpb}{\overline{\Bbb Z_p}}

\newcommand{\Qpb}{\overline{\Bbb Q_p}}

\newcommand{\Fpb}{\overline{\Bbb F_p}}

\newcommand{\ben}{\begin{enumerate}}
\newcommand{\een}{\end{enumerate}}

\newcommand{\C}{{\mathcal C}}
\newcommand{\D}{{\mathcal D}}

\newcommand{\E}{\mathcal{E}}

\begin{document}

\title[Reductions of tensor categories modulo primes]
{Reductions of tensor categories modulo primes}
\author{Pavel Etingof}
\address{Department of Mathematics, Massachusetts Institute of Technology,
Cambridge, MA 02139, USA} \email{etingof@math.mit.edu}

\author{Shlomo Gelaki}
\address{Department of Mathematics, Technion-Israel Institute of
Technology, Haifa 32000, Israel} \email{gelaki@math.technion.ac.il}

\date{\today}


\dedicatory{Dedicated to Miriam Cohen}

\begin{abstract}
We study good (i.e., semisimple) reductions of semisimple rigid
tensor categories modulo primes. A prime $p$ is called good for a
semisimple rigid tensor category $\C$ if such a reduction exists
(otherwise, it is called bad). It is clear that a good prime must be
relatively prime to the M\"uger squared norm $|V|^2$ of any simple
object $V$ of $\C$. We show, using the Ito-Michler theorem in finite
group theory, that for group-theoretical fusion categories, the
converse is true. While the converse is false for general fusion
categories, we obtain results about good and bad primes for many
known fusion categories (e.g., for Verlinde categories). We also
state some questions and conjectures regarding good and bad primes.
\end{abstract}

\maketitle

\section{Introduction}

In this paper we study good (i.e., semisimple) reductions of
semisimple rigid tensor categories (in particular, fusion
categories) modulo primes. Namely, let $\C$ be a semisimple rigid
tensor category over $\Qb$ (recall that any fusion category over a
field of characteristic zero is defined over $\Qb$ by the Ocneanu
rigidity theorem, see \cite[Theorem 2.28]{ENO1}) and let $p$ be a
prime. We say that $p$ is a {\bf good prime} for $\C$ if there is a
semisimple rigid tensor category $\C_p$ over $\Fpb$ (with the same
Grothendieck ring as $\C$) which admits a lift to a tensor category
over $\Qpb$ (in the sense of \cite[Section 9.2]{ENO1}) that becomes
equivalent to $\C$ after extension of scalars from $\Qb$ to $\Qpb$.
In this case, the category $\C_p$ is called {\bf a good reduction}
of $\C$ modulo $p$. Otherwise, if $\C_p$ does not exist, we say that
$p$ is a {\bf bad prime} for $\C$.

It is not hard to show that for any fusion category $\C$, there is a
finite set of bad primes. The goal of this paper is to find
conditions for a prime to be bad for $\C$, and determine all such
primes for various examples of fusion categories.

The organization of the paper is as follows. In Section 2, we
discuss general properties of good and bad primes. In Section 3 we
determine the bad primes for group-theoretical categories, in
particular for representation categories of finite groups, using the
Ito-Michler theorem in finite group theory; namely, we show that a
prime is bad if and only if it divides the dimension of a simple
object. In Section 4, we discuss bad primes for the Verlinde
categories (i.e., fusion categories coming from quantum groups at
roots of unity). Finally, in Section 5 we discuss some questions and
conjectures regarding good and bad primes.

{\bf Acknowledgments.} We are very grateful to Noah Snyder for
useful discussions, in particular for contributing Example
\ref{sny}. The research of the first author was partially supported
by the NSF grant DMS-1000113. The second author was supported by The
Israel Science Foundation (grant No. 317/09). Both authors were
supported by BSF grant No. 2008164.

\section{Good and bad primes}

Let $p$ be prime. Let $\Zpb$ denote the ring of integers in the
field $\Qpb$ (the algebraic closure of ${\Bbb Q}_p$).
Let $\mathfrak{m}$ be the maximal ideal in $\Zpb$.
Clearly, $\Zpb/\mathfrak{m}\cong \Fpb$.

\begin{definition}
Let $\C$ be a semisimple rigid tensor category over $\Qb$. A {\bf
good reduction} of $\C$ modulo $p$ is a semisimple rigid tensor
category $\C_p$ over $\Fpb$, categorifying the Grothendieck ring of
$\C$, such that there is a lift of $\C_p$ to $\Zpb$ (i.e., a
semisimple rigid tensor category $\widetilde{\C}_p$ over $\Zpb$
which yields $\C_p$ upon reduction by the maximal ideal
$\mathfrak{m}$) and a tensor equivalence
$\widetilde{\C}_p\otimes_{\Zpb}\Qpb\cong \C\otimes_{\Qb}\Qpb$. If a
good reduction of $\C$ exists, we will say that $p$ is a {\bf good
prime} for $\C$. Otherwise we will say that $p$ is a {\bf bad prime}
for $\C$.
\end{definition}

\begin{example} A pointed fusion category
$\Vect_G^\omega$, where $G$ is a finite group and $\omega$ is a
$3-$cocycle of $G$ with values in $\Qbs$, has a good reduction
modulo any prime, since $\omega$ can be chosen to take values in
roots of unity.
\end{example}

\begin{remark}
Note that two non-equivalent fusion categories can have equivalent
good reductions modulo a prime $p$. E.g., for any $3-$cocycle
$\omega$ on $\Bbb Z/p\Bbb Z$, the category $\Vect_{\Bbb Z/p\Bbb
Z}^\omega$ reduces to $\Vect_{\Bbb Z/p\Bbb Z}$ (with trivial
cocycle) in characteristic $p$ (as $\omega$ can be chosen to take
values in $p-$th roots of unity). This cannot happen, however, if
the global dimension of either of these two categories is relatively
prime to $p$ (see \cite[Theorem 9.6]{ENO1}).
\end{remark}

The following proposition gives a necessary condition for a prime to
be good. Recall (see e.g., \cite{Mu,ENO1}) that for any simple
object $V$ of $\C$, one can define its {\em M\"uger's squared norm}
$|V|^2$, which is an algebraic integer. Also recall that two
algebraic integers are called \emph{relatively prime} if their norms
(which are usual integers) are relatively prime.

\begin{proposition}\label{if}
Let $\C$ be a semisimple rigid tensor category over $\Qb$. If $p$ is
a good prime for $\C$ then $p$ must be relatively prime to $|V|^2$
for any simple object $V\in \C$.
\end{proposition}

\begin{proof} If $p$ is not relatively prime to $|V|^2$ then $|V|^2$
would have to be zero in the reduction $\C_p$. But M\"uger's squared
norm of any simple object of a semisimple rigid tensor category must
be nonzero (see \cite{Mu,ENO1}).
\end{proof}

\begin{remark}
One may ask if the converse of Proposition \ref{if} holds. It turns
out that the answer is negative in general
(see Example \ref{sny} below).
However, in the next section we will prove that the
answer is positive for group-theoretical categories.
\end{remark}

Recall that a fusion category $\C$ is called \emph{pseudounitary} if
the M\"uger's squared norm $|V|^2$ of every simple object $V$
coincides with the square ${\rm FPdim}(V)^2$ of its Frobenius-Perron
dimension. Recall also that every weakly integral fusion category
$\C$ (i.e., such that ${\rm FPdim}(\C)$ is an integer) is
pseudounitary.

\begin{corollary} If $\C$ is a pseudounitary fusion category
then any good prime $p$ for $\C$ is relatively prime to the
Frobenius-Perron dimension ${\rm FPdim}(V)$ for any simple object
$V\in \C$. \qed
\end{corollary}

\begin{proposition}\label{finman}
For any fusion category $\C$, there are finitely many bad primes.
\end{proposition}

\begin{proof}
If we write the structure morphisms of $\C$ in some basis,
there will be only finitely many primes in the denominator,
and all the other primes are automatically good.
\end{proof}

\begin{remark}
Note that Proposition \ref{finman} is not true for infinite
semisimple rigid tensor categories. For example, if $\C$ is the
category of representations of ${\rm SL}(2,\Qb)$ then all primes are
bad for $\C$, since M\"uger's squared norm of the $n-$dimensional
representation $V_{n-1}$ of ${\rm SL}(2,\Qb)$ is $n^2$.
\end{remark}

Recall (\cite{DGNO}, \cite{ENO2}) that a fusion category is called
\emph{weakly group-theoretical} if it is obtained by a chain of
extensions and equivariantizations from the trivial category.

\begin{proposition}
If $\C$ is a weakly group-theoretical category
of Frobenius-Perron dimension $d$ then any bad prime
$p$ for $\C$ divides $d$.
\end{proposition}

\begin{proof}
The proof is by induction on the length of the chain.
The base of induction is clear, and the induction step follows
from the following lemma.

\begin{lemma}
Let $\D$ be a fusion category of global dimension $n$, and let $G$
be a finite group. Let $\C$ be either a $G-$extension of $\D$, or a
$G-$equivariantization of $\D$. Let $p$ be a prime, which is
relatively prime to both $n$ and $|G|$. If $\D$ admits a good
reduction modulo $p$, then so does $\C$.
\end{lemma}

\begin{proof}
Let $\D_p$ be a good reduction of $\D$ modulo $p$. Since $p$ is
relatively prime to $n$, this is a non-degenerate fusion category
(i.e., its global dimension $\ne 0$), so the lifting theory of
\cite[Section 9]{ENO1}, applies. In particular, it follows from the
proof of \cite[Theorem 9.6]{ENO1} that the lifting map ${\rm
Eq}(\D_p)\to {\rm Eq}(\D)$ between the groups of tensor
auto-equivalences of $\D_p$ and $\D$, defined in \cite[Section
9]{ENO1}, is an isomorphism (i.e., any auto-equivalence of $\D$ has
a reduction modulo $p$). This implies that we have an isomorphism of
the corresponding categorical groups $\underline{\rm Eq}(\D_p)\to
\underline{\rm Eq}(\D)$ (see \cite[Section 4.5]{ENO3}). Since
$G-$actions on a fusion category $\E$ correspond to homomorphisms
$G\to \underline{\rm Eq}(\E)$, we find that $G-$actions on $\D$ and
$\D_p$ (and hence the corresponding equivariantizations) are in
bijection. Moreover, since $|G|$ is coprime to $p$, it is easy to
see that all the $G-$equivariantizations of $\D_p$ are semisimple.
Thus, any $G-$equivariantization of $\D$ has a good reduction modulo
$p$ (which is the corresponding $G-$equivariantization of $\D_p$).

Also, we see that the lifting map defines an isomorphism of
Brauer-Picard groups ${\rm BrPic}(\D_p)\to {\rm BrPic}(\D)$ (see
\cite[Section 4.1]{ENO3}). This follows from the fact that by
\cite[Theorem 1.1]{ENO3}, for any fusion category $\E$ the
Brauer-Picard group ${\rm BrPic}(\E)$ is naturally isomorphic to the
group ${\rm EqBr}(\mathcal{Z}(\E))$ of braided auto-equivalences of
the Drinfeld center $\mathcal{Z}(\E)$, and from the proof of
\cite[Theorem 9.6]{ENO1}. This implies that we have an isomorphism
of Brauer-Picard groupoids $\underline{\underline{{\rm
BrPic}(\D_p)}}\to \underline{\underline{{\rm BrPic}(\D)}}$ (see
\cite[Section 4.1]{ENO3}). Since for any fusion category $\E$,
$G-$extensions of $\E$ are classified by homomorphisms $G\to
\underline{\underline{{\rm BrPic}(\E)}}$, we conclude that
$G-$extensions of $\D_p$ and $\D$ are in bijection. Thus, any
$G-$extension of $\D$ has a good reduction modulo $p$ (which is the
corresponding $G-$extension of $\D_p$).
\end{proof}
The proposition is proved.
\end{proof}

\section{Bad primes for group-theoretical categories}

Let $G$ be a finite group and let $\omega\in Z^3(G,\Qbs)$ be a
normalized $3-$cocycle on $G$. Let $\Vect_G^{\omega}$ be the fusion
category of finite-dimensional $G-$graded $\Qb-$vector spaces with
the associativity defined by $\omega$. Let $H$ be a subgroup of $G$
such that $\omega|_{H}$ is cohomologically trivial, and let $\psi\in
C^2(H,\Qbs)$ be such that $\omega|_{H} = d\psi$; then the twisted
group algebra $\Qb^{\psi}[H]$ is a unital associative algebra in
$\Vect_G^{\omega}$. The {\em group-theoretical} category $\C =
\C(G,H,\omega,\psi)$ is defined as the fusion category of
$\Qb^{\psi}[H]-$bimodules in $\Vect_G^{\omega}$ (see
\cite[Definition 8.40]{ENO1}, \cite{O}).

Let $R$ be a set of representatives of double cosets of $H$ in $G$.
In \cite{O} it is explained that there is a bijection between the
isomorphism classes of simple objects $V_{g,\rho}$ in $\C$ and
isomorphism classes of pairs $(g,\rho)$, where $g \in R$ and $\rho$
is an irreducible projective representation of $H^g:=H\cap gHg^{-1}$
with a certain $2-$cocycle $\psi^g$. (See also \cite{GN}.) The
dimension of $V_{g,\rho}$ is $\frac{|H|}{|H^g|} \dim(\rho)$.

Assume now that $H\subseteq G$ contains an abelian subgroup $A$
which is normal in $G$. Let $K:=H/A$, and assume that the orders of
$K$ and $A$ are coprime (so that $H$ is isomorphic to $A\rtimes K$).

\begin{proposition}\label{na}
Suppose that $\psi^g|_A$ is cohomologically trivial for any $g\in
G$. Then there exist cocycles $\eta\in Z^2(G/A,A^\vee)$,
$\widetilde{\omega}\in Z^3(\widetilde{G}_\eta,\Qbs)$ and a
$2-$cochain $\widetilde{\psi}\in C^2(K,\Qbs)$ such that
$\C\cong\C(\widetilde{G}_\eta,K,\widetilde{\omega},\widetilde{\psi})$,
where the crossed product
$\widetilde{G}_\eta:=A^\vee\rtimes_{\eta}G/A$ is the extension of
$G/A$ by $A^\vee$ using $\eta$.
\end{proposition}

\begin{remark}
Note that since the order of $K$ is relatively prime to the order of
$A$, we have an embedding of $K$ into $\widetilde{G}_\eta$ defined
canonically up to conjugation. In Proposition \ref{na}, we can use
any such embedding.
\end{remark}

\begin{proof}
In the special case $K=1$, this is \cite[Theorem 4.9]{N}. This
theorem claims that there exist $\eta$ and $\widetilde{\omega}$ such
that $\C(G,A,\omega,\psi)$ is equivalent to
$\C(\widetilde{G}_\eta,\{1\},\widetilde{\omega},1)$. Now consider
the module category ${\mathcal M}(H,\psi)$ over
$\C(G,A,\omega,\psi)$. It is shown by a direct computation that
under the above equivalence, this module category goes to a module
category of the form ${\mathcal M}(K,\widetilde{\psi})$. Passing to
the dual categories, we get the statement of the proposition.
\end{proof}

\begin{example}
Consider the symmetric group on three letters $S_3$. Then $\Rep(S_3)
= \mathcal{C}(S_3, S_3, 1, 1) = \mathcal{C}(S_3, S_2, 1, 1)$, which
shows that $3$ is a good prime for $\Rep(S_3)$.
\end{example}

We will use the following deep theorem, whose proof relies on the classification
of finite simple groups, in an essential way.

\begin{theorem}\label{im} (Ito-Michler, \cite{Mi}, see also \cite{It})
Let $G$ be a finite group and let $p$ be a prime number dividing
the order of $G$. Suppose that $p$ does not divide
the dimension of any irreducible representation of $G$.
Then the Sylow $p-$subgroup $S$ of $G$ is both normal and abelian, so
(by a theorem of Schur)
$G=S\rtimes K$, where the order of $K$ is not divisible by $p$.
\end{theorem}

We can now state and prove the main result of this section.

\begin{theorem}\label{goodprime}
Let $\C=\C(G,H,\omega,\psi)$ be a group-theoretical category.
A prime number $p$ is bad for $\C$ if and only if $p$ divides the
dimension of some simple object of $\C$.
\end{theorem}

\begin{proof} The ``if" direction follows from Proposition \ref{if}.

Let us prove the ``only if" direction.
Assume that $\omega$ and $\psi$ take values in roots of unity
(which we can do without loss of generality). Take a prime $p$.
If $p$ does not divide the order of $H$, the category
$\C_p:=\C(G,H,\omega_p,\psi_p)$ (where $\omega_p,\psi_p$ are the
reductions of $\omega,\psi$ modulo $p$) is
semisimple over $\overline{\mathbb{F}}_p$, so $p$ is good for $\C$.

Therefore, it suffices to consider the case when $p$ divides the
order of $H$. Suppose that $p$ does not divide the dimension of any
simple object of $\C$, and let us show that $p$ is good. Since
$\Rep(H)$ is a fusion subcategory of $\C$, Theorem \ref{im} implies
that $H=A\rtimes K$, where the order of $K$ is not divisible by $p$
and $A$ is an abelian normal $p-$subgroup of $H$.

We claim that in fact  $A$ is normal in $G$, not only in $H$.
Indeed, by our assumption, $p$ does not divide $|H|/|H^g|$, which
implies that the order of $H^g$ is divisible by the order of $A$, so
$H^g$ must contain $A$. So for any $g\in G$ we have
$g^{-1}Ag\subseteq A\rtimes K$, hence $g^{-1}Ag=A$ (since $K$ is a
$p'-$group).

Finally, we claim that $\psi^g|_A$ is trivial. Indeed, a projective
representation of $A$ with $2-$cocycle $\psi^g|_A$ is a simple
object of $\C$, so its dimension must be coprime to $p$. But the
dimension of this object is a power of $p$, which is $1$ only if
$\psi^g$ is trivial, as desired.

Now, by Proposition \ref{na}, $\C$ is equivalent to
$\C(\widetilde{G}_\eta,K,\widetilde{\omega},\widetilde{\psi})$.
We can assume that $\widetilde{\omega},\widetilde{\psi}$ take
values in roots of unity. Since $p$ does
not divide the order of $K$, the category
$\C(\widetilde{G}_\eta,K,\widetilde{\omega},\widetilde{\psi})$ admits a
reduction $\C(\widetilde{G}_{\eta_p},K,\widetilde{\omega}_p,
\widetilde{\psi}_p)$ to a
fusion category in characteristic $p$, so $p$ is good for $\C$.
\end{proof}

\begin{corollary}\label{repg} Let $G$ be a finite group.
A prime number $p$ is good for $\Rep(G)$ if and only if $p$ does not
divide the dimension of any irreducible representation of $G$. \qed
\end{corollary}

\begin{example}\label{pgl} Corollary \ref{repg} fails for
reductive algebraic groups. Indeed, consider the group $G:={\rm
PGL}(2,\Qb)$, and let $\C:=\Rep(G)$ be the category of rational
representations of $G$. Its simple objects $V_{2m}$ have dimensions
$2m+1$, where $m$ is a nonnegative integer. Thus, the prime $2$ is
the only prime that has a chance to be good for this category, since
any other prime divides the dimension (and hence, M\"uger's squared
norm) of some irreducible object. Nevertheless, we claim that the
prime $2$ is bad for $\Rep(G)$ (so in fact, for this category, all
primes are bad). Indeed, suppose we have a good reduction $\C_2$ of
$\Rep(G)$ mod $2$, and let us derive a contradiction.

Let $X\in \C_2$ be the reduction of the $3-$dimensional (adjoint)
representation of $G$. Then we have a unique morphism $m: X\otimes
X\to X$ and isomorphism $b: X\to X^*$ up to scaling. Thus, if we fix
the scaling of these, we can attach an amplitude $A(T)$ to any
planar trivalent graph (allowing multiple edges but no self-loops),
by contracting the morphisms $m$ corresponding to the vertices using
$b$. Moreover, the amplitude $A(T_2)$ of the simplest planar
trivalent graph $T_2$ (with $2$ vertices and $3$ edges, i.e., a
circle with a diameter) is nonzero. Let us normalize it to be $1$.
Then we get a well defined amplitude of any other planar trivalent
graph.

Consider the amplitude $A(T_4)$ of the next simplest graph $T_4$
with $4$ vertices - the square with diagonals (with one of the
diagonals going outside the square to avoid intersection and realize
the graph in the plane). It is supposed to be the reduction mod $2$
of the corresponding amplitude $A_{\Qb}(T_4)$ over $\Qb$. But over
$\Qb$, this is a computation in the representation theory of the Lie
algebra ${\mathfrak{sl}}(2)$. Namely, if $x_i$ is an orthonormal
basis of ${\mathfrak{sl}}(2)$, we get
\begin{equation}\label{lie}
\sum x_ix_jx_ix_j|_{V_2}=A_\Qb(T_4)\left(\sum x_i^2\right)^2|_{V_2},
\end{equation}
where $V_2$ is the adjoint representation. A direct computation then
shows that $A_\Qb(T_4)$ is $3/2$, which is a contradiction since
$A(T_4)$ is supposed to be the reduction of this number modulo $2$.
\footnote{For the reader's convenience, let us give explicit
expressions of $A(T_2)$ and $A(T_4)$ in terms of $m$ and $b$. Let
$m_*: X\to X\otimes X$ be the map obtained from the dual of $m$ by
identifying $X^*$ with $X$ using $b$. Then we have
$$
A(T_2)=\Tr(mm_*),\ A(T_4)=\Tr(m(1\otimes m)(m_*\otimes 1)m_*).
$$
Using that if $X$ is the adjoint representation of
${\mathfrak{sl}}(2)$ then $m$ is the Lie bracket and $m_*$ is its
dual under the Killing form, we arrive at formula (\ref{lie}).}
\end{example}

\section{Verlinde categories}

Let $\mathfrak{g}$ be a simple complex Lie algebra. For simplicity
let us assume that it is simply laced (so $(\alpha,\alpha)=2$ for
roots). Let $h$ be the Coxeter number of $\mathfrak{g}$, let
$\theta$ be the highest root of $\mathfrak{g}$, and let $\rho$ be
half the sum of positive roots of $\mathfrak{g}$.

Let $l>h$ be a positive integer, and let $q\in \mathbb{C}$ be such
that the order of $q^2$ is $l$. Set
$[n]_q:=\frac{q^n-q^{-n}}{q-q^{-1}}$.

Following Andersen and Paradowski \cite{AP},
one can define the Verlinde fusion category
$\C(\mathfrak{g},q)$. Its simple objects are $V_\lambda$,
where $\lambda$ are dominant weights for $\mathfrak{g}$ such that
$(\lambda+\rho,\theta)<l$. The dimension of $V_\lambda$ is given by
the $q-$Weyl formula:
$$
\dim (V_\lambda)=\prod_{\alpha>0}
\frac{[(\lambda+\rho,\alpha)]_q}{[(\rho,\alpha)]_q}=
q^{2(\lambda,\rho)}\prod_{\alpha>0}\frac{
(1-q^{-2(\lambda+\rho,\alpha)})}{(1-q^{-2(\rho,\alpha)})}.
$$

We will need the following elementary and well known lemma (whose
proof we include for the reader's convenience).

\begin{lemma}\label{norm}
Let $v$ be a primitive root of unity of order $n$. The norm $N(1-v)$
of $1-v$ is given by:
$$N(1-v)=
\left\{ \begin{array}{cc}1, &\mbox{$n$ is not a prime power} \\
p, &\mbox{$n$ is a power of a prime $p$}
\end{array}
\right. .$$
\end{lemma}

\begin{proof}
Let $n=\prod_{i=1}^m p_i^{s_i}$ be the prime factorization of $n$.
Then
$$
N(1-v)=\prod_{k=1,\,p_1,\dots,p_m\nmid k}^n(1-v^k),
$$
which is
the value of the cyclotomic polynomial $\prod_{k=1,\,p_1,\dots,p_m\nmid
k}^n(x-v^k)$ at $x=1$. By the exclusion-inclusion principle,
\begin{eqnarray*}
\lefteqn{ \prod_{k=1,\,p_1,\dots,p_m\nmid k}^n(x-v^k)|_{x=1}}\\
& & = \frac{(x^n-1)\prod_{i<j}(x^{\frac{n}{p_ip_j}}-1)\cdots}
{\prod_i(x^{\frac{n}{p_i}}-1)\cdots}|_{x=1}=
\frac{n\prod_{i<j}\frac{n}
{p_ip_j}\cdots}{\prod_i\frac{n}{p_i}\cdots}.
\end{eqnarray*}
Now, the power of $p_i$ on the right hand side equals
$$
1-(m-1)+{m-1 \choose 2}-\dots=(1-1)^{m-1}=0
$$
unless $m=1$, in which
case it equals $1$, as claimed.
\end{proof}

\begin{theorem}
Assume that $l$ is odd.

(i) If $l$ is a prime then any prime $p$ is good for
$\C(\mathfrak{g},q)$.

(ii) If $l$ is not a prime then a prime $p\ge h$ is good for
$\C(\mathfrak{g},q)$ if and only if $p$ does not divide $l$.
\end{theorem}

\begin{remark}
Note that the condition $p\ge h$ cannot be dropped in (ii).
For instance, assume that ${\mathfrak{g}}={\mathfrak{sl}}(n)$
(so $h=n$), and let $l=h+1$ (which we assume to be odd).
Then $\C(\mathfrak{g},q)$ is a pointed category, so
it has no bad primes.
\end{remark}

\begin{proof}
If $p$ is relatively prime to $l$,
then one can define the fusion category $\C_p(\mathfrak{g},q)$
over $\overline{\mathbb{F}_p}$ (similarly to \cite{AP}),
so $p$ is a good prime.

(i) For $p=l$, there is a symmetric category over
$\overline{\mathbb{F}_p}$ which lifts to the braided category
$\C(\mathfrak{g},q)$ (see \cite{GM1}, \cite{GM2}, \cite{AP}), so
again $p$ is good.

(ii) Suppose $p$ divides $l$. By Proposition \ref{if}, it is enough
to show that at least one of the numbers $\dim (V_\lambda)$ is not
relatively prime to $p$ or, equivalently, that its $p-$adic norm is
$<1$. To this end, pick $\lambda:=(\frac{l}{p}-1)\rho$. Such
$\lambda$ is allowed since
$(\lambda+\rho,\theta)=\frac{(h-1)l}{p}<l$.

By Lemma \ref{norm}, the $p-$adic norm of $1-v$ is $1$ if the order
of $v$ is not a prime power, and is $p^{-1/p^s(p-1)}$ if the order
of $v$ is equal to $p^{s+1}$. Now, the order of
$q^{-2(\lambda+\rho,\alpha)}=q^{-2l(\rho,\alpha)/p}$ is $p$ for any
$\alpha$ (since by assumption, $(\rho,\alpha)<h\le p$). Thus, the
$p-$adic norm of every factor in the numerator of the $q-$Weyl
formula for $\dim (V_\lambda)$ is $p^{-1/(p-1)}$, while the $p-$adic
norm of every factor of the denominator is at least that, and the
$p-$adic norm of $[(\rho,\alpha_i)]_q=1$ is $1$ for any simple root
$\alpha_i$. Thus, the $p-$adic norm of $\dim (V_\lambda)$ is $<1$,
i.e., $\dim(V_\lambda)$ is not relatively prime to $p$, as desired.
\end{proof}

\begin{remark}
If $l$ is a prime then the dimensions $\dim (V_\lambda)$ are units.
Indeed, for any $1\le s<l$ there is a Galois automorphism sending $q^{2}$
to $q^{2s}$, so the norm of
$[s]_q=\frac{q^s-q^{-s}}{q-q^{-1}}$ is
equal to $1$.
\end{remark}

\begin{example}\label{sny} The following example, due to Noah Snyder,
shows that the converse to Proposition \ref{if} fails for general
(and even for braided) fusion categories. Indeed, consider the
Verlinde category $\C(\mathfrak{sl}_2,q)$, where $q:=e^{\pi i/8}$
(i.e., $l=8$), with simple objects $V_0=\bold 1$, $V_1$,\dots,
$V_6$, and let $\D$ be its tensor subcategory generated by
$V_0,V_2,V_4,V_6$. Then $\dim (V_0)=\dim(V_6)=1$,
$\dim(V_2)=\dim(V_4)=1+\sqrt{2}$, which are all units, so any prime
is relatively prime to the dimension of any simple object of $\D$.
Yet, considering the graph $T_4$ as in Example \ref{pgl}, we get by
a direct calculation
$$
A_\Qb(T_4)=\frac{[3]_q([3]_q-2)}{[3]_q-1}=\frac{1}{\sqrt{2}},
$$
which shows that $2$ is a bad prime for $\D$.

Also, as was explained to us by Noah Snyder, it follows from the
arguments similar to those in \cite{MS} (computation of the ``third
twisted moment'') that 3 is a bad prime for the $6-$object Haagerup
category, even though the dimensions of all its simple objects are
units.\footnote{As was explained to us by Noah Snyder, it can also
be shown, using a specific presentation of the $6-$object Haagerup
category (see \cite{Iz}), that any other prime is good for this
category.}
\end{example}

\section{Conjectures and questions}

\begin{question}
If a prime $p$ is relatively prime to the global dimension
$\dim(\C)$, does $p$ have to be a good prime for $\C$ in cases
(i)-(iii)?

(i) $\C$ is a general fusion category.

(ii) $\C$ is an integral fusion category, i.e.,
$\C=\Rep(H)$ for a semisimple quasi-Hopf algebra $H$.

(iii) $\C=\Rep(H)$ for a semisimple Hopf algebra $H$.
\end{question}

Note that in (ii) and (iii), $\dim(\C)=\dim(H)$.

A positive answer in case (iii) would imply that any prime divisor
of the dimension of a simple $H-$module divides the dimension of
$H$, which is a weak (but still open) form of Kaplansky's $6-$th
conjecture.

\begin{question}
(i) Let $\C$ be an integral fusion category, and suppose that $p$
does not divide the (Frobenius-Perron) dimension of any simple
object. Does $p$ have to be a good prime for $\C$, i.e, does Theorem
\ref{goodprime} hold for $\C$?

(ii) Is this true for weakly group-theoretical categories?
\end{question}

\begin{question} Does any fusion category $\C$ admit at most one
good reduction modulo any prime $p$?
\end{question}

\begin{remark}
The answer is ``yes'' if $p$ is relatively prime to the global
dimension of $\C$, by \cite[Theorem 9.6]{ENO1}.
\end{remark}

\begin{question} Suppose that $p$ is a good prime for $\C$. Is $p$ good
for any module category over $\C$?
\end{question}

\end{document}